\documentclass{amsart}
\usepackage{latexsym}
\usepackage{amsfonts}
\newtheorem{theorem}{Theorem}[section]
\newtheorem{lemma}[theorem]{Lemma}
\newtheorem{corollary}[theorem]{Corollary}
\newtheorem{proposition}[theorem]{Proposition}

\theoremstyle{definition}
\newtheorem{definition}[theorem]{Definition}

\theoremstyle{remark}
\newtheorem{remark}[theorem]{Remark}
\numberwithin{equation}{section}

\newcommand{\sai} {\mbox{$\to \kern -0.50 em \to$}}
\newcommand{\nsai} {\mbox{$\not\to \kern -0.50 em \to$}}

\begin{document}
\title[SOME SIMPLIFICATIONS IN THE PRESENTATION OF COMPLEX POWER SERIES]  
{SOME SIMPLIFICATIONS IN THE PRESENTATIONS OF COMPLEX POWER SERIES AND UNORDERED SUMS} 

\author{Oswaldo Rio Branco de Oliveira}

\begin{abstract} {This text provides very easy and short proofs of some  basic properties of complex power series (addition, subtraction, multiplication, division, rearrangement, composition, differentiation, uniqueness, Taylor's series, Principle of Identity, Principle of Isolated Zeros, and Binomial Series). This is done by simplifying the usual presentation of unordered sums of a (countable) family of complex numbers. All the proofs avoid formal power series, double series, iterated series, partial series, asymptotic arguments, complex integration theory, and uniform continuity. The use of function continuity as well as epsilons and deltas is kept to a mininum.}
\end{abstract}
 
\maketitle

\hspace{- 0,6 cm}{\sl Mathematics Subject Classification: 30B10, 40B05, 40C15, 40-01, 97I30, 97I80}

\hspace{- 0,6 cm}{\sl Key words and phrases:} Power Series, Multiple Sequences, Series, Summability, Complex Analysis, Functions of a Complex Variable.

\par

\tableofcontents
 
\section{Introduction}  
The objective of this work is to provide a simplification of the theory of unordered sums of a family of complex numbers (in particular, for a countable family of complex numbers) as well as very easy proofs of basic operations and properties concerning complex power series, such as addition, scalar multiplication, multiplication, division, rearrangement, composition, differentiation (see Apostol ~\cite{APO1} and Vyborny ~\cite{VY}), Taylor's formula, principle of isolated zeros, uniqueness, principle of identity, and binomial series. We achieve our goal regarding complex power series by avoiding the following five approaches commonly used to present their basic operations and properties: (1) the one that relies on the relatively elaborated classical summability concept of an unordered family (countable or not) of vectors/numbers (see Beardon ~\cite[pp.~67--147]{BE-LIM}, Browder ~\cite[pp.~47--97]{BRO}, Morrey and Protter ~\cite[pp.~219--262]{MOPRO}, and Hirsch and Lacombe ~\cite[pp.~127--130]{HILA}); (2) the formal power series approach (see Cartan ~\cite[pp.~9--27]{CA} and Lang ~\cite[pp.~37--86]{LA}); (3) the very usual methods that employ some sort of combination of double series, iterated series, partial series, and uniform continuity (see Apostol ~\cite[pp.~371--416]{APO2} and Rudin ~\cite[pp.~69--203]{RU}); (4) the classical approach given by Knopp ~\cite[pp.~151--433]{KK}, which leads to some  quite harsh and long argumentations and calculations and employs sub-series (also called partial series in ~\cite{APO2}); and (5) the standard method that provides elementary proofs of some properties of the power series and then employs the powerful Cauchy's Integral Formula and the basic properties of holomorphic functions to derive the difficult properties of power series (see Agarwal, Perera, and Pinelas ~\cite[pp.~151--168]{APP}, Bak and Newman ~\cite[pp.~25--90]{BN}, Boas ~\cite[pp.~18--117]{BO}, Burckel ~\cite[pp.~53--190]{BU}, Busam and Freitag ~\cite[pp.~109--124]{BF}, Conway ~\cite[pp.~30--44]{CON}, Gamelin ~\cite[pp.~130--164]{GA}, and Narasimhan and Nievergelt ~\cite[pp.~3--51]{NN}). Instead of utilizing these approaches, this paper shows a simplification of the proofs of the mentioned power series operations and properties by introducing a definition of an unordered sum of a countable family of complex numbers that is equivalent to the classical definition as applied to a countable family of vectors in a complete vector space, but it is easier to manipulate in the context of complex numbers. It is important to emphasize that the definition introduced in this text is easily extendable to a sum of an uncountable family of complex numbers.

It is interesting to notice that modern authors such as Burckel ~\cite{BU}, Lang ~\cite{LA}, Newman and Bak ~\cite{BN}, 
Remmert ~\cite[pp.~109--132]{RR}, and others  stress the importance of the study of complex power series. However, many books overlook the proofs of some of these important operations and properties of power series.

The presentation of this text is very basic. In fact, we also avoid the use of the following concepts: uniform continuity, compactness, connectedness, and asymptotic expansion. Furthermore, we keep the use of epsilons and deltas to a minimum. 

The proofs in Section 7 (Power Series - Algebraic Properties) employ  neither function continuity nor complex differentiability, whereas the proofs in Section 8 (Power Series - Analytic Properties) employ at least one of these concepts.

Since every complex power series is absolutely convergent in its disk of convergence, we will focus our attention on absolutely convergent series.

\section{Preliminaries}
Let us indicate by $\mathbb N=\{0,1,2,...\}$ the set of all natural numbers,
$\mathbb Q$ the field of rational numbers, $\mathbb R$ the complete field of
the real numbers, and $\mathbb C$ the algebraically closed field of the complex
numbers. Moreover, if $z\in\mathbb C$ then we write $z=x+iy$, where $x=$
Re$(z)\in \mathbb R$ is the real part of $z$, $y=$ Im$(z)\in \mathbb R$ is the imaginary part of $z$ and $i^2=-1$. Given $z=x+iy$ in $\mathbb C$, its
conjugate is the complex number $\overline{z} = x-iy$ and its absolute value
is the non-negative real number $|z|=\sqrt{z\overline{z}}=\sqrt{x^2 + y^2}$. 

The open disk centered at a point $z_0\in \mathbb C$ with radius $r>0$ is
the set $D(z_0;r)=\{z \in \mathbb C: |z-z_0|<r\}$. Similarly, the compact disk centered at $z_0$ with radius $r\geq 0$ is the set
$\overline{D}(z_0;r)=\{z \in \mathbb C: |z-z_0|\leq r\}$. 

We say that a set $S$ is \textit{countable} if either $S$ is finite or $S$ can be put into one-to-one correspondence with $\mathbb N$. A set $S$ is countably infinite, or \textit{denumerable}, if $S$ is countable but not finite. We will use the following well-known properties: (1) every nonempty subset of $\mathbb N$ is countable; (2) every subset of a countable set is countable; (3) the union of a countable family of countable sets is countable; (4) the finite cartesian product of countable sets is a countable set.

Given $X\subset \mathbb C$, a point $p \in \mathbb C$ is an \textit{accumulation point} of $X$ if every disk $D(p;r)$, where $r>0$, contains a point of $X$ distinct of $p$.

\section{Absolutely Convergent Series and Commutativity}

\begin{definition}\label{DEF31} Given a sequence of complex numbers $(z_n)_{n\in \mathbb N}$, let us consider a series $\sum_{n=0}^{+\infty}z_n$ of complex numbers and $(s_n)_{n\in \mathbb N}$, where $s_n=z_0+ z_1 + \cdots +z_n$, the sequence of partial sums of the series. We say that the series is
\begin{itemize}
\item[$\circ$] \textit{convergent}, with sum equal $z\in \mathbb C$ (we write $\sum_{n=0}^{+\infty} z_n=z$), if the sequence $(s_n)$ converges to $z$.
\item[$\circ$] \textit{absolutely convergent} if the series $\sum_{n=0}^{+\infty}|z_n|$ is convergent.
\item[$\circ$] \textit{conditionally convergent} if it is convergent but not absolutely convergent.
\item[$\circ$] \textit{commutatively convergent} if every rearrangement (reordering) $\sum_{n=0}^{+\infty}z_{\sigma(n)}$, where $\sigma:\mathbb N\to \mathbb N$ is a bijection, is a convergent series (we will soon see that, in this case, all the rearrangements of $\sum_{n=0}^{+\infty} z_n$ have equal sums).
\item[$\circ$] \textit{divergent} if it is not convergent.
\end{itemize}
\end{definition}

A series $\sum_{n=0}^{+\infty}z_n$ is called divergent if it is not convergent. We also denote an arbitrary series $\sum_{n=0}^{+\infty}z_n$ by 
\[\sum\limits^{+\infty} z_n.\]

We say that the complex series $\sum_{n=0}^{+\infty}z_n$ is generated by the sequence $(z_n)$.

Let us consider a real series $\sum_{n=0}^{+\infty}x_n$ and its sequence of partial sums $(s_n)$.
If $\sum_{n=0}^{+\infty}x_n$ is convergent, then we write  $\sum_{n=0}^{+\infty}x_n <\infty$. If $s_n\rightarrow \pm\infty$ as $n\to +\infty$, then we write $\sum_{n=0}^{+\infty}x_n=\pm \infty$.  

\begin{remark}\label{REM32} Given a series $\sum_{n=0}^{+\infty}p_n$ (convergent or otherwise) of non-negative numbers $p_n\geq 0$, with sequence of partial sums $(s_n)$, we have  
\[\lim_{n\to +\infty}s_n= \sum_{n=0}^{+\infty}p_n\in [0,+\infty].\]
\end{remark}

\begin{theorem}\label{TEO33} Let $\sum_{n=0}^{+\infty}p_n$ be a series (convergent or otherwise) of non-negative numbers and $\sigma:\mathbb N\to \mathbb N$ be a bijection. Then,
\[\sum_{n=0}^{+\infty}p_n=\sum_{n=0}^{+\infty}p_{\sigma(n)}.\]
\end{theorem}
\begin{proof} Evaluating the limit as $N \to +\infty$, where $N\in \mathbb N$,  of the trivial inequality
\[\sum_{j=0}^Np_j\leq \sum_{n=0}^{+\infty}p_{\sigma(n)},\]
we see that $\sum_{j=0}^{+\infty} p_j\leq \sum_{n=0}^{+\infty}p_{\sigma(n)}$. Vice-versa, we have $ \sum_{n=0}^{+\infty}p_{\sigma(n)} \leq \sum_{j=0}^{+\infty} p_j$.
\end{proof}

\begin{definition}\label{DEF34} Let $\sum_{n=0}^{+\infty}a_n$ be a series of real numbers.  The \textit{positive and negative parts of $a_n$} are, respectively,
\begin{displaymath}
p_n\ =\ \left\{\begin{array}{ll}
a_n,\ &\textrm{if}\ a_n\ \geq \ 0\\
0,\ &\textrm{if}\ a_n\ \leq \ 0\\
\end{array}\ \ \ \ \ \ \ \ \textrm{and}\ \ \ \ \ \  q_n\ =\ \left\{\begin{array}{ll}
\ \ 0,\ &\textrm{if}\ a_n\ \geq \ 0\\
- a_n,\ &\textrm{if}\ a_n\ \leq \ 0.\\
\end{array}
\right.
\right.
\end{displaymath}
Given any $n\in \mathbb N$ we have 
\begin{displaymath}
(3.4.1)\ \ \ \ \  
\left\{\begin{array}{ll}
0\leq p_n \leq |a_n|\\
0\leq q_n \leq |a_n|
\end{array}
\right. , \ \ \ \ \ 
\left\{\begin{array}{ll}
a_n = p_n - q_n\\
|a_n| = p_n + q_n
\end{array}
\right., \ \ \ \ \textrm{and}\ \ \   \ \left\{\begin{array}{ll}
p_n = \frac{|a_n|+a_n}{2}\\ 
q_n = \frac{|a_n|-a_n}{2}.\ \ \ \ 
\end{array}
\right.
\end{displaymath}
\end{definition}
Keeping the notation in Definition \ref{DEF34} we have the next result on real series.

\begin{theorem}\label{TEO35} Let $\sum_{n=0}^{+\infty} a_n$ be an arbitrary series of real numbers. The following are true.
\begin{itemize} 
\item[(a)] $\sum_{n=0}^{+\infty} |\,a_n| = \sum_{n=0}^{+\infty} p_n + \sum_{n=0}^{+\infty} q_n$.
\item[(b)] If $\sum_{n=0}^{+\infty}a_n$ converges absolutely, then it also converges commutatively. Moreover, the series generated by the sequences $(p_n)$ and $(q_n)$ both converge and
\[  \ \sum_{n=0}^{+\infty} a_n = \sum_{n=0}^{+\infty} p_n -\sum_{n=0}^{+\infty}q_n .\] 
Furthermore, all the rearrangements of $\sum_{n=0}^{+\infty}a_n$ have equal sums.
\item[(c)] If $\sum_{n=0}^{+\infty} a_n $ converges conditionally, then we have $\sum_{n=0}^{+\infty}p_n=\sum_{n=0}^{+\infty}q_n=+\infty$.
\item[(d)]If $\sum_{n=0}^{+\infty}a_n$ converges conditionally, then there exists a divergent rearrangement of the series $\sum_{n=0}^{+\infty}a_n$.
\end{itemize}
\end{theorem}
\begin{proof} Let us employ relations (3.4.1).
\begin{itemize} 
\item[(a)] From Remark \ref{REM32} it follows that    
\[\sum_{n=0}^{+\infty}|a_n|= \lim\limits_{m\to +\infty}\sum\limits_{n = 0}^{m} |a_n|  =  \lim_{m\to +\infty}\left[\sum\limits_{n = 0}^{m}p_n + \sum\limits_{n = 0}^{ m} q_n\right] = \sum_{n=0}^{+\infty}|p_n| + \sum_{n=0}^{+\infty}|q_n| .\]
\item[(b)] Since $0\leq p_n\leq |a_n|$ and $0\leq q _n\leq |a_n|$, we deduce that the series $\sum_{n=0}^{+\infty}p_n$ and
$\sum_{n=0}^{+\infty} q_n$ are convergent and then, by Theorem  \ref{TEO33}, both series are commutatively convergent. Therefore, we can conclude that the series given by the subtraction $\sum_{n=0}^{+\infty} p_n-\sum_{n=0}^{+\infty}q_n=\sum_{n=0}^{+\infty}a_n$ also is commutatively convergent. In fact, supposing that $\sigma:\mathbb N\to \mathbb N$ is a bijection, from Remark \ref{REM32} we obtain
\[\sum\limits^{+\infty}a_{\sigma(n)}= \sum\limits^{+\infty}p_{\sigma(n)} -\sum\limits^{+\infty}q_{\sigma(n)}= \sum\limits^{+\infty}p_n-\sum\limits^{+\infty}q_n=\sum\limits^{+\infty}a_n .\]
Thus, all the rearrangements of $\sum_{n=0}^{+\infty}a_n$ have equal sums.
\item[(c)] Since the series $\sum_{n=0}^{+\infty} a_n$ converges, from the identity $a_n=p_n-q_n$ we deduce that $\sum_{n=0}^{+\infty}p_n$ converges if and only if $\sum_{n=0}^{+\infty} q_n$ converges. However, by hypothesis we have  $\sum_{n=0}^{+\infty}|a_n|=+\infty$. Hence, from (a) it follows that at least one of the series $\sum_{n=0}^{+\infty}p_n$ and $\sum_{n=0}^{+\infty}q_n$ diverges. Thus, the series generated by the sequences $(p_n)$ and $(q_n)$ are both divergent.
\item[(d)] From (c) it follows that $\sum_{n=0}^{+\infty} p_n = \sum_{n=0}^{+\infty} q_n = +\infty$. Hence, we rearrange the series $\sum_{n=0}^{+\infty}a_n$ in the following way: at step 1, we collect the first terms $a_n\geq 0$, where $n\in \mathbb N$, whose sum is strictly bigger than $1$; at step 2, we collect the first terms $a_n<0$, where $n\in \mathbb N$, whose sum with the previously collected terms is strictly smaller than $0$; at step 3, having subtracted from $\mathbb N$ all the indices already selected, we collect the next non-negative terms $a_n\geq 0$, where $n\in \mathbb N$, whose sum with the previously collected terms is strictly bigger than $1$. Iterating this argument, we obtain a rearrangement of the original series. The sequence of partial sums of this rearrangement admits a subsequence with all terms strictly bigger than $1$, and  another subsequence with all terms strictly negative. Consequently, this rearrangement diverges.
\end{itemize}
\end{proof}

\begin{corollary}\label{COR36}  Let $\sum_{n=0}^{+\infty}a_n$ be a real series. The following are equivalent.
\begin{itemize}
\item[(a)] The series is commutatively convergent.
\item[(b)] The series is absolutely convergent.
\item[(c)] The series $\sum_{n=0}^{+\infty}p_n$ and $\sum_{n=0}^{+\infty} q_n$ are both convergent.
\end{itemize}
\end{corollary}
\begin{proof}Let us split up the proof into three parts.
\begin{itemize} 
\item[(a)]$\Rightarrow$(b) From Theorem \ref{TEO35} (d) it follows that $\sum_{n=0}^{+\infty}a_n$ is not conditionally convergent. Hence, since $\sum_{n=0}^{+\infty}a_n$ converges, the series $\sum_{n=0}^{+\infty}|a_n|$ converges. 
\item[(b)]$\Rightarrow$(a) Follows from Theorem \ref{TEO35} (b).
\item[(b)]$\Leftrightarrow$(c)  Follows from Theorem \ref{TEO35} (a).
\end{itemize}
\end{proof}

\begin{remark}\label{REM37} Considering a complex series $\sum_{n=0}^{+\infty}z_n$, the following assertions hold. 
\begin{itemize}
\item[$\circ$] The complex series is commutatively convergent if and only if the real series $\sum_{n=0}^{+\infty}\textrm{Re}(z_n)$ and $\sum_{n=0}^{+\infty}\textrm{Im}(z_n)$ are both commutatively convergent.
\item[$\circ$] The inequalities right below are true, for every $n$ in $\mathbb N$,
\[(3.7.1)   0\leq |\textrm{Re}(z_n)|\leq |z_n|,   \   0\leq |\textrm{Im}(z_n)|\leq|z_n|,  \ \textrm{and}  \ |z_n|\leq |\textrm{Re}(z_n)| + |\textrm{Im}(z_n)| . \]
\end{itemize}
\end{remark}

Using Remark \ref{REM37} we obtain the next result.

\begin{corollary}\label{COR38} Let $\sum_{n=0}^{+\infty}z_n$ be a complex series. The following are equivalent.
\begin{itemize}
\item[(a)] The series is absolutely convergent.
\item[(b)] The series $\sum_{n=0}^{+\infty} \textrm{Re}(z_n)$ and $\sum_{n=0}^{+\infty} \textrm{Im}(z_n)$ are both absolutely convergent.
\item[(c)] The series is commutatively convergent.
\end{itemize}
\end{corollary}
\begin{proof} It follows from Remark \ref{REM37} and Corollary \ref{COR36}.
\end{proof}

\section{Unordered Countable Sums and Commutativity}

In this section, we present a definition of the value of a series of complex numbers that does not depend on the order of its terms. This will be useful to define the sum of a countable family in $\mathbb C$ and, in particular, the sum of a sequence in $\mathbb C$.

\begin{theorem}\label{TEO41} Let $\sum_{n=0}^{+\infty} p_n$ be a series of non-negative terms. Then, we have 
\[ \sum_{n=0}^{+\infty} p_n  = \rho,  \ \textrm{with}\  \rho = \sup\left\{ \sum\limits_{n\in F}\,p_n: F\subset \mathbb N\ \textrm{and}\ F\ \textrm{finite} \right\} \in  [0,+\infty] .\] 
\end{theorem}
\begin{proof} Let $(s_m)$ be the (increasing) sequence of partial sums of the given series and $F$ be an arbitrary nonempty finite subset of $\mathbb N$. From the inequalities
\[ \sum_{m\in F}p_m \leq s_{\max(F)}\leq \sum_{n=0}^{+\infty}p_n \ \textrm{and} \ s_m=\sum_{k\in \{0,1,\ldots ,m\}} p_k\leq \rho \]
we deduce that $\rho =\sup\{\sum_{m\in F}p_m: F \textrm{ is a finite subset of}\ \mathbb N\} \leq\sum_{n=0}^{+\infty}p_n = \lim\limits_{m\to +\infty} s_m  \leq \rho$.
\end{proof}

Henceforth, $\mathbb K$ is either the field $\mathbb R$ or the field $\mathbb C$.

\begin{definition}\label{DEF42} Let $J$ be an arbitrary countable index set and $\mathbb K$ fixed.
\begin{itemize}
\item[$\circ$] A \textit{family} in $\mathbb K$, indexed by $J$, is a function $x: J \to \mathbb K$. We denote this family by $(x_j)_{_J}$, where $x_j = x(j)$ is the $j$-term of the family, for each $ j$ in $J$. We also denote the family $(x_j)_J$ as, briefly, $(x_j)$.
\item[$\circ$] Given two families $(x_j)_{_J}$ and $(y_j)_{_J}$, and $\lambda$ in $\mathbb K$, we define the addition and the scalar multiplication as, respectively,
\[ (x_j)_{_J} + (y_j)_{_J}=(x_j + y_j)_{_J}\ \ \textrm{and}\ \ \lambda(x_j)_{_J}=(\lambda x_j)_{_J}.\]
\end{itemize}
\end{definition}

Clearly, every sequence is a family. From now on $J$ and $L$ denote countable index sets. If $J=\mathbb N\times \mathbb N = {\mathbb N}^2$,  then $x:\mathbb N\times \mathbb N \to \mathbb K$ is a double sequence.

\vspace{0,2 cm}

Based on Theorem \ref{TEO41} we present the following notations and definitions.

\begin{definition}\label{DEF43} Given a family $(p_{\,j})_{j\in J}$ of non-negative terms, we put 
\[ \sum p_j = \sup\left\{ \sum\limits_{j\in F}\,p_j: F\subset J\ \textrm{and}\ F\ \textrm{finite}  \right\}  .\]
\begin{itemize}
\item[$\circ$]  If $ \sum p_j  < \infty $, then we say that $(p_j)_J$ is a \textit{summable family} (or, briefly, summable) with sum $\sum p_{\,j}$, also denoted by $\sum_{J} p_{\,j}$ or $\sum_{j \in J}  p_{\,j}$.
\item[$\circ$] If $J=\mathbb N$ and $(p_n)_{\mathbb N}$ is summable, then we call $(p_n)$ a \textit{summable sequence}.
\end{itemize}
\end{definition}

\begin{remark}\label{REM44} Let us consider $(p_j)_J$ a countably infinite family of non-negative terms.
\begin{itemize}
\item[$\circ$] Let  $\sigma:L\to J$ be a bijection. It is trivial to verify that  
\[\sum_J p_j=\sum_L p_{\sigma(l)}.\]
Thus, $(p_j)_J$ is summable if and only if $(p_{\sigma(l)})_L$ is summable. 
\item[$\circ$] If $J=\mathbb N$, then from Theorem \ref{TEO41} we may conclude that $\sum p_n=\sum_{n=0}^{+\infty}p_n$.
\end{itemize}
\end{remark}

\begin{corollary}\label{COR45} Let $(p_n)$ be a  sequence of non-negative terms. Then, $(p_n)$ is summable if and only if $\sum_{n=0}^{+\infty} p_n$ is convergent. If $(p_n)$ is summable, then we have
\[\sum p_n\ =\ \sum\limits^{+\infty}p_n.\]
\end{corollary} 
\begin{proof} It can be deduced from Remark \ref{REM44}.
\end{proof}

\begin{corollary}\label{COR46} Let $(p_{\,j})_J$ be a countably infinite family of non-negative terms and $\sigma:\mathbb N \to J$ be a bijection. Then, $(p_{\,j})_J$ is a summable family if and only if the series $\sum_{n=0}^{+\infty} p_{\sigma(n)}$ is convergent. If $(p_j)$ is summable, then we have 
\[\sum p_{\,j}\ =\ \sum\limits^{+\infty}p_{\sigma(n)}.\] 
\end{corollary}
\begin{proof}  It can be deduced from Remark \ref{REM44} and Corollary \ref{COR45}. 
\end{proof}

\begin{remark} Given a series $\sum_{n=0}^{+\infty}p_n$ of non-negative terms, either convergent or divergent (in such a case, diverging to $+\infty$), the identity $\sum_{n=0}^{+\infty} p_n =\sum p_n$ holds. Thus, for a series of non-negative terms we can freely use the notation $\sum p_n$ to denote a series (see Definition \ref{DEF31}) or an unordered sum (see Definition \ref{DEF43}).  
\end{remark}

\begin{definition}\label{DEF48} Let $J$ be a countable indexing set.
\begin{itemize}
\item[$\circ$] A family $(x_j)_J$ of real numbers is \textit{summable} if the families $(p_j)_J$ and $(q_j)_J$ of the positive and negative parts of $x_j$, $j \in J$ (see Definition \ref{DEF34}), respectively, are summable. If $(x_j)_J$ is summable, then its \textit{unordered sum} is 
\[\sum x_j  = \sum p_{\,j}  - \sum q_{\,j}.\]
\item[$\circ$] A family $(z_j)_J$ of complex numbers is \textit{summable} if the families $\big(\textrm{Re}(z_j)\big)_J$ and $\big(\textrm{Im}(z_j)\big)_J$ of the real and imaginary parts of $z_j$, with $j\in J$,  respectively, are summable. If $(z_j)_J$ is summable, then its \textit{unordered sum} is
\[\sum z_j = \sum\textrm{Re}(z_j) + i \sum\textrm{Im}(z_j) .\]
\item[$\circ$] A family $(z_j)_J$, either in $\mathbb R$ or in $\mathbb C$, is \textit{absolutely summable} if the family $(\,|z_j|\,)_J$  is summable. In other words, $(z_j)_J$ is absolutely summable if
\[\sum|z_j| < \infty .\]
\end{itemize}
\end{definition}

\begin{remark} ~\label{REM49} We also denote $\sum z_j$ by the notations $\sum_J z_j$ and $\sum_{j\in J}z_j$.
\end{remark}

\begin{lemma}\label{LE410} Let $(z_j)_J$ be a denumerable family in $\mathbb C$ and $\sigma:\mathbb N\to J$ be a bijection. The following statements are equivalent:
\begin{itemize}
\item[(a)] $(z_j)_J$ is summable.
\item[(b)] $(\,z_j\,)_J$ is absolutely summable.
\item[(c)] There exists $M>0$ such that $\sum_{j\in F} |z_j|\leq M$, for all finite subset $F\subset J$.
\item[(d)] $\sum_{n=0}^{+\infty}|z_{\sigma(n)}|$ is (absolutely) convergent.
\end{itemize}
Morever, if any of the statements (a), (b), (c) or (d) holds, then
\[\sum z_j =  \sum\limits_{n=0}^{+\infty}z_{\sigma(n)} . \]
\end{lemma}
\begin{proof} Let us consider the real families $\big(\textrm{Re}(z_j)\big)_J$ and $\big(\textrm{Im}(z_j)\big)_J$. We also consider the families of their positive parts  $(p_j)_J$ and $(P_j)_J$, respectively, and the families of their negative parts $(q_j)_J$ and $(Q_j)_J$, respectively. 
\begin{itemize}
\item[(a)]$\Leftrightarrow$(b) By employing 3.4.1 and 3.7.1, we see that for every $j\in J$ we have  
\[ 0 \leq   \max\big\{p_j,q_j,P_j,Q_j\big\} \leq  |z_j| \leq  p_j+q_j+P_j+Q_j .\]  
Thus, the sum $\sum |z_j|$ is finite if and only if the four sums $\sum p_j$, $\sum q_j$, $\sum P_j$, and $\sum Q_j$ are finite. Therefore, the family $(|z_j|)_J$ is summable if and only if the family $(z_j)_J$ is summable.
\item[(b)]$\Leftrightarrow$(c) It is straightforward.
\item[(b)]$\Leftrightarrow$(d) It follows from Corollary \ref{COR46}.
\end{itemize}
Finally, let us suppose that at least one of (a), (b), (c) or (d) is true. From the Definition \ref{DEF48} we can deduce the identities
$\sum z_j = \sum\textrm{Re}(z_j) + i\sum\textrm{Im}(z_j)$, $\sum\textrm{Re}(z_j)= \sum p_j - \sum q_j$, and $\sum\textrm{Im}(z_j)= \sum P_j - \sum Q_j$. Hence, employing Corollary \ref{COR46} we obtain
\begin{displaymath}
\begin{array}{ll}
\sum p_j =\sum\limits_{n=0}^{+\infty}p_{\sigma(n)},\   \sum q_j = \sum\limits_{n=0}^{+\infty}q_{\sigma(n)}, \    \sum P_j = \sum\limits_{n=0}^{+\infty}P_{\sigma(n)},  \  \textrm{and} \    \sum Q_j=\sum\limits_{n=0}^{+\infty}Q_{\sigma(n)} .
\end{array}
\end{displaymath}
Combining these four identities we conclude that
\begin{displaymath}
\begin{array}{ll}
\sum z_j &=\left[\sum\limits_{n=0}^{+\infty} p_{\sigma(n)}\, -\, \sum\limits_{n=0}^{+\infty}q_{\sigma(n)}\right] \ +\ i\left[\sum\limits_{n=0}^{+\infty} P_{\sigma(n)}\, -\, \sum\limits_{n=0}^{+\infty}Q_{\sigma(n)}\right]\\
\\
&= \sum\limits_{n=0}^{+\infty}\textrm{Re}[z_{\sigma(n)}]\, + \, i\sum\limits_{n=0}^{+\infty}\textrm{Im}[z_{\sigma(n)}] = \sum\limits_{n=0}^{+\infty}z_{\sigma(n)}.
\end{array}
\end{displaymath}
\end{proof}

\begin{remark} Lemma \ref{LE410} implies that the definition of a summable denumerable family employed in this text is equivalent to the classical definition, as applied to a denumerable family of complex numbers (see Beardon ~\cite[pp.~67--68]{BE-LIM}, Browder ~\cite[pp.~47--57]{BRO}, Morrey and Protter ~\cite[pp.~241--262]{MOPRO}, and Hirsch and Lacombe ~\cite[p.~127]{HILA}). In fact, both definitions are equivalent to Lemma ~\ref{LE410} (d).  
\end{remark}

\begin{corollary}\label{COR412} Let $(z_j)_J$ be a summable family and $L$ be a subset of $J$. Then, the family $(z_l)_{l\in L}$ is also summable.
\end{corollary}
\begin{proof} The case where $L$ is finite is obvious. Let us suppose that $L$ is denumerable. Hence, $J$ is also denumerable. From Lemma \ref{LE410} (b) it follows that $\sum_{j\in J}|z_j|<\infty$. Thus, since $L\subset J$, we have $\sum_{l\in L}|z_l|\leq \sum_{j\in J}|z_j|<\infty$.
Finally, from Lemma \ref{LE410} (a) we conclude that the family $(z_l)_L$ is summable.
\end{proof}

\begin{proposition}\label{PRO413} Let us consider $\mathbb K$ fixed. Let $(a_j)_J$ and $(b_j)_J$ be summable families in $\mathbb K$, and $\lambda \in \mathbb K$. Then, the families $(a_j + b_j)_J$ and $(\lambda a_j)_J$ are summable and the following properties are satisfied.
\begin{itemize}
\item[(a)] $ \sum (a_j + b_j)  = \sum a_j  + \sum b_j$.
\item[(b)] $\sum \lambda a_j = \lambda \sum a_j$.
\item[(c)] $\big|\sum a_j\big|\leq \sum |a_j|$.
\end{itemize} 
\end{proposition}
\begin{proof} The case where $J$ is finite is obvious. Hence, supposing that $J$ is countably infinite, let us consider an arbitrary bijection $\sigma: \mathbb N\to J$.

\begin{itemize}
\item[(a)] From Lemma \ref{LE410} it follows that the series $\sum_{n=0}^{+\infty}|a_{\sigma(n)}|$ and $\sum_{n=0}^{+\infty}|b_{\sigma(n)}|$ are both convergent. Hence, by the triangle inequality we conclude that the series $\sum_{n=0}^{+\infty} |a_{\sigma(n)}+b_{\sigma(n)}|$ also converges. Therefore, employing Lemma \ref{LE410} we deduce that the family $(a_j+b_j)_J$ is summable and satisfies 
\[ \sum_J (a_j + b_j) = \sum\limits_{n=0}^{+\infty} [a_{\sigma(n)}\,+\,b_{\sigma(n)}] =  \sum\limits_{n=0}^{+\infty}a_{\sigma(n)}\,+\, \sum\limits_{n=0}^{+\infty}b_{\sigma(n)} = \sum_J a_j \, +\, \sum_J b_j .\]
\item[(b)] It is trivial.
\item[(c)] We already showed that $\sum_{n=0}^{+\infty} a_{\sigma(n)}$ converges absolutely. Therefore, from a well-known property of series we have $\big|\sum_{n=0}^{+\infty} a_{\sigma(n)}\big|\leq \sum_{n=0}^{+\infty} |a_{\sigma(n)}|$. Hence, from Lemma \ref{LE410} we conclude that $\big|\sum a_j|\leq \sum |a_j|$.
\end{itemize}
\end{proof}

\section{Unordered Countable Sums and Associativity.}

In this section we show that given the sum of an arbitrary countable family of non-negative numbers $\sum_Jp_j$, we can freely associate them. It is also shown that we can freely dissociate each $p_j$, where $j\in J$, as a countable family of non-negative numbers. Furthermore, given a summable countable family of complex numbers, we prove that we can freely associate them. These operations keep the value of the respective sum unchanged.

To simplify the presentation, we enunciate the results and give the respective proofs for the case where the family is a sequence (i.e., the family is indexed by $\mathbb N$). These results extend easily to families indexed by an arbitrary denumerable set $J$. To see this, one just needs to employ a bijection $\sigma:\mathbb N\to J$.

Before continuing, we introduce a notation. Let $J$ be a nonempty set.
 
\begin{itemize}
\item[$\circ$] A \textit{partition} of $J$ is a collection of nonempty subsets $\{J_{\,l} : J_l\subset J \ \textrm{and}\ l\in L\}$, where $L$ is an index set, such that the sets $J_l$, with $l\in L$, are \textit{pairwise disjoint} (i.e., $J_{\,l}\cap J_{\,l'}=\emptyset$, for any $l\in L$ and $l'\in L$ such that $l\neq l'$) and 
$$J= \bigcup_{l \in L}J_{\,l}.$$
\end{itemize}
\begin{theorem}\label{TEO51} Let $(p_n)$ be a sequence of non-negative numbers. Let $\mathbb N = \bigcup_{l \in L}J_{\,l}$ be an arbitrary partition of $\mathbb N$. Then we have 
\[ \sum p_n = \sum_{l \,\in \,L}\,\sum_{n\,\in J_{\,l}}p_n,\]
where we write $\sum_{l \,\in \,L}\sum_{n\,\in J_{\,l}}p_n =+\infty$ when $\sum_{n\,\in J_{\,l}}p_n=+\infty$ for some $l\in L$.
\end{theorem}
\begin{proof} Let us analyze two cases.
\begin{itemize}
\item[(1)] Let us suppose that $\sum_{n\,\in J_{\,l'}}p_n=+\infty$ for some $l'\in L$. From Definition \ref{DEF43} we obtain the inequality $\sum p_n\geq \sum_{n\in J_{\,l'}}p_n$. Thus, we have $\sum p_n=+\infty$.

\item[(2)] By supposing that $\sum_{n\,\in J_{\,l}}p_n<\infty$, for all $ l\in L$, let us show two inequalities. Given $F\subset \mathbb N$, with $F$ finite, by hypothesis there exists a finite subset of distinct indices $ \{l_1, \ldots,  l_k\}\subset L$ such that $F\subset J_{l_1}\cup \ldots \cup J_{l_k}$. It then follows 
\[\sum\limits_{n\in F}p_n \ \leq \ \sum\limits_{J_{l_1}}p_n + \cdots +\,\sum\limits_{J_{l_k}}p_n\ \leq \ \sum\limits_{l \in \,L}\sum\limits_{n\,\in J_{\,l}}p_n.\]
As a result, by the definition of $\sum p_n$ we obtain the first inequality 
\[\sum p_n  \ \leq \ \sum\limits_{l \in \,L}\sum\limits_{n\,\in J_{\,l}}p_n .\]
In order to obtain the reverse inequality, we notice that given a finite set of distinct indices $\{l_1, \ldots, l_k\}\subset L$ then the sets $J_{l_1},  \ldots, J_{l_k}$ are pairwise disjoint. Hence, given for each $r\in \{1, \ldots, k\}$ an arbitrary finite set $F_{l_r} \subset J_{l_r}$, we deduce that the sets $F_{l_1}, \ldots, F_{l_k}$ are also pairwise disjoint. Thus,
\[\sum\limits_{F_{l_1}} p_n+ \cdots + \sum\limits_{ F_{l_k}} p_n\ \leq\ \sum p_n .\]
Considering the inequality right above, fixing the sets $F_{l_2}, \ldots, F_{l_k}$ in it, and then taking in it the supremum over the family of all finite sets $F_{l_1}$ contained in $J_{l_1}$ we arrive at the inequality $\sum_{J_{l_1}} p_n + \sum_{F_{l_2}}p_n + \cdots +\sum_{F_{l_k}}p_n\leq \sum p_n$. Now, considering this last inequality, fixing the sets $F_{l_3}, \ldots, F_{l_k}$ in it, and then taking in it the supremum over the family of all finite sets $F_{l_2}$ contained in $J_{\,l_2}$ we arrive at $\sum_{J_{l_1}} p_n + \sum_{J_{l_2}}p_n +\sum_{F_{l_3}}p_n+ \cdots +\sum_{F_{l_k}}p_n\leq \sum p_n$. Thus, by induction we find the inequality
\[\sum\limits_{J_{l_1}} p_n + \sum\limits_{J_{l_2}}p_n + \cdots +\sum\limits_{J_{l_k}}p_n\ \leq \ \sum p_n .\]
Finally, since $\{l_1, l_2, \ldots, l_k\}$ is an arbitrary finite subset of $L$ we obtain 
\[\sum_{l\in L}\sum_{n\in J_{\,l}}p_n \, \leq \, \sum p_n.\]
\end{itemize}
\end{proof}

\begin{corollary} \label{COR52}({\sf Associative Law}) 
Let $(a_n)_{ \mathbb N}$ be a summable complex sequence and $ \bigcup_{l \in L}J_{\,l}$ be a partition of $\mathbb N$. Then the family $(a_n)_{n\in J_l}$ is summable, for all $l \in L$, and
\[ \sum a_n = \sum_{L}\sum_{n\in J_{\,l}}a_n.\]
\end{corollary}
\begin{proof} From Corollary \ref{COR412} we deduce that $(a_n)_{n\in J_l}$ is summable, for all $l$ in $L$. Let us analyze two cases.
\begin{itemize}
\item[(1)]
If $(a_n)\subset \mathbb R$, then the claimed identity follows from the decomposition $\sum\limits a_n = \sum\limits p_n - \sum\limits q_n$ (see  Definition \ref{DEF48}), Theorem \ref{TEO51}, and Proposition \ref{PRO413}. 

\item[(2)] If $(a_n)\subset \mathbb C$, then the claimed identity follows from the decomposition $\sum a_n = \sum \textrm{Re}(a_n) + i\sum\textrm{Im}(a_n)$ (see  Definition \ref{DEF48}), case (1), and Proposition \ref{PRO413}.
\end{itemize}
\end{proof}

\

\section{Sum of a Double Sequence and The Cauchy Product}
Given a complex double sequence  $(a_{(n,m)})_{\,\mathbb N\times \mathbb N}$, we denote its terms by $a_{nm}$ or $a_{(n,m)}$, where $(n,m)\in \mathbb N\times \mathbb N$, and its unordered sum (if it exists) by 
\[\sum\limits_{\mathbb N\times \mathbb N}a_{nm}\,,\ \sum\limits_{n,m}a_{nm}\ \textrm{or}\ \sum a_{nm}.\]
It is also usual to denote the double sequence by the infinite matrix
\begin{displaymath}
\left(\begin{array}{cccc}
a_{11} & a_{12}& a_{13}& \ldots\\
a_{21} & a_{22} & a_{23} &\ldots\\
a_{31}  & a_{32} & a_{33}&\ldots\\
\ldots& \ldots& \ldots & \ldots\\
\ldots & \ldots & \ldots & \ldots \\
\end{array}\right) .
\end{displaymath}

\newpage

\begin{proposition}\label{PRO61} Let us suppose that $\sum |\,a_{nm}|<\infty $. Then, $(a_{nm})$ is summable and the following are true.
\begin{itemize}
\item[(a)] If $\bigcup_{l\in L} J_l$ is a partition of $\mathbb N\times \mathbb N$, then
\[ \sum_{l\in L}\sum_{(n,m)\in J_l}a_{nm}= \sum a_{nm} .\]
\item[(b)] Given any bijection $\sigma : \mathbb N \to \mathbb N\times \mathbb N$, we have 
$\sum\limits_{\,k\,=0}^{+\infty} a_{\sigma(k)}\,= \sum a_{nm}$.
\end{itemize}
\end{proposition}
\begin{proof} The summability of $(a_{nm})$ follows from Lemma \ref{LE410}.
\begin{itemize}
\item[(a)] It follows from Corollary \ref{COR52}.
\item[(b)] It follows from Lemma \ref{LE410}.
\end{itemize}
\end{proof}

\begin{definition} \label{DEF62} Given $\sum_{n=0}^{+\infty} a_n$ and $\sum_{m=0}^{+\infty} b_m$, their Cauchy product is the series
\[ (a_0b_0) + (a_0b_1+ a_1b_0) + (a_0b_2 + a_1b_1 + a_2b_0)+\cdots  = \sum_{p=0}^{+\infty}\, c_p,\ \ \textrm{where} \ c_p = \sum_{n + m \,= p}a_nb_m  .\]
\end{definition}

\begin{corollary}\label{COR63} Let us consider $\sum_{n=0}^{+\infty} a_n =a$ and $\sum_{m=0}^{+\infty} b_m =b$ two absolutely convergent series. Then, we have $\sum|\,a_nb_m|< \infty$ and the following properties:
\begin{itemize}
\item[(a)] $\sum\limits a_nb_m = ab$.
\item[(b)] Their Cauchy product is an absolutely convergent series and 
\[ \sum_{p=0}^{+\infty}\Big(\sum_{n+m=p}a_n\, b_m\Big)  =  ab. \]
\end{itemize}
\end{corollary}
\begin{proof} From Theorem ~\ref{TEO51} we deduce that
$\sum|a_nb_m| = \big(\sum |a_n|\big)\,\big(\sum |b_m|\big)<\infty\,$.

\begin{itemize}
\item[(a)] Employing Proposition ~\ref{PRO61} (a) and Lemma 4.10, in this order, we obtain 
\[\sum a_nb_m= \sum_{n\in \mathbb N}\sum_{m\in \mathbb N}a_nb_m= \sum\limits_{n=0}^{+\infty}a_n\Big(\sum\limits_{m=0}^{+\infty}b_m\Big)= ab.\]
\item[(b)] Let us pick the partition $(J_p)_{ p \in \mathbb N}$ of $\mathbb N \times \mathbb N$, where $J_p = \{(n,m) \in \mathbb N\times \mathbb N: n + m = p\}$. By item (a), Theorem ~\ref{PRO61} (a), and Lemma ~\ref{LE410} we have 
\[ab = \sum a_n b_m  = \sum_{p\in \mathbb N}c_p =\sum_{p=0}^{+\infty}c_p,  \ \textrm{where}\ c_p=\sum_{n+m=p}a_nb_m \ \textrm{and} \ \sum\limits_{p=0}^{+\infty} |c_p|<\infty. \]
\end{itemize}
\end{proof}

\section{Power Series - Algebraic Properties}

\begin{definition}\label{DEF71}  A \textit{power series} with coefficients $(a_n)_{\mathbb N}\subset \mathbb C$ and centered at $z_0\in \mathbb C$ is a function of the form $f(z)=\sum_{n=0}^{+\infty}a_n(z-z_0)^n$, where $z$ is a complex variable. 
\end{definition}

We say that the power series $\sum_{n=0}^{+\infty}a_n(z-z_0)^n$  converges (diverges) at a point $z=w\in \mathbb C$ if the numerical series $\sum_{n=0}^{+\infty}a_n(w-z_0)^n$ converges (diverges). Through the translation $w=z-z_0$ we go from the power series $\sum_{n=0}^{+\infty}a_n(z-z_0)^n$ to the power series $\sum_{n=0}^{+\infty}a_nw^n$. Hence, we simplify this presentation by supposing the power series centered at $z_0=0$.

\begin{theorem}\label{TEO72} Let us consider a power series $\sum_{n=0}^{+\infty}a_n z^n$ and
\[\rho=\sup\Big\{ r\geq 0: \sum_{n=0}^{+\infty} a_nz^n \ \textrm{is convergent for some }\ |z|=r\Big\}\,, \ \textrm{with} \ \rho \in [0,+\infty].\]
The following are true.
\begin{itemize}
\item[(a)] If $|z|<\rho$, then the power series converges absolutely at $z$.
\item[(b)] If $|z|>\rho$, then the power series diverges at $z$.
\end{itemize}
\end{theorem}
\begin{proof} Since $\sum_{n=0}^{+\infty}|a_n|r^n$ converges at $r=0$, we conclude that $\rho$ is well defined.
\begin{itemize}
\item[(a)] By the definition of $\sup$ there exists $w\in \mathbb C$, with $|z|<|w|<\rho$, such that $\sum_{n=0}^{+\infty}a_nw^n<\infty$. In such a case, by a well-known property of series we conclude that there exists $M\in \mathbb R$ satisfying $|a_nw^n|\leq M$, for all $n\in \mathbb N$. Hence, $\sum_{n=0}^{+\infty}|a_n||z|^n= 
\sum_{n=0}^{+\infty}|a_n||w|^n(|z|/|w|)^n\leq M\sum_{n=0}^{+\infty}\Big(|z|/|w|)^n$. Since the geometric series $\sum_{n=0}^{+\infty}(|z|/|w|)^n$ converges, the series  $\sum_{n=0}^{+\infty}|a_n||z|^n$ also does.
\item[(b)] It is trivial.
\end{itemize}
\end{proof}

\begin{remark} Given a subset $Z\subset \mathbb C$, we say that a sequence of functions $f_n: Z \to \mathbb C$, where $n\in \mathbb N$, converges uniformly to a function $f:Z\to \mathbb C$ if given any $\epsilon >0$, then there exists $N=N(\epsilon)\in \mathbb N$ such that for all $n\geq N$ and all $z\in Z$ we have $|f_n(z) -f(z)|< \epsilon$. Keeping the hypothesis in Theorem \ref{TEO72}, it can be shown that the power series $\sum_{n=0}^{+\infty}a_nz^n$ converges uniformly within $\overline{D}(0;r)$, for all $0<r<\rho$ (see any book in the references).     
\end{remark}

\begin{definition}\label{DEF74} Let us consider $\sum_{n=0}^{+\infty} a_nz^n$ and $\rho$ as in Theorem \ref{TEO72}.
\begin{itemize}
\item[$\circ$] The \textit{radius of convergence} of the power series is $\rho$.
\item[$\circ$] The (open) \textit{disk of convergence} of $\sum_{n=0}^{+\infty} a_nz^n$ is $D(0;\rho)$, if $0<\rho<\infty$. 
\end{itemize}
\end{definition}

\begin{remark} Let $\sum_{n=0}^{+\infty}a_nz^n$ be a power series with disk of convergence $D(0;\rho)$, where $\rho > 0$, and $z\in D(0;\rho)$. By Theorem \ref{TEO72}, the series $\sum_{n=0}^{+\infty}a_nz^n$ converges absolutely. Hence, the sequence $(a_nz^n)_{n \in \mathbb N}$ is summable and $\sum a_nz^n =\sum_{n=0}^{+\infty} a_nz^n$. In such a case, we also write the power series as (briefly) $\sum a_nz^n$. 
\end{remark}

Given two power series $\sum a_n z^n$ and $\sum b_nz^n$ both convergent inside $D(0;r)$, with $r>0$, it is clear that their sum $\sum(a_n+b_n)z^n=\sum a_nz^n + \sum b_nz^n$ is a power series that converges in $D(0;r)$. Moreover, given $\lambda\in \mathbb C$, the power series $\sum \lambda a_nz^n=\lambda \sum a_nz^n$ also converges inside $D(0;r)$.

The next result is called the rearrangement theorem on power series and claims that every power series can be developed as a power series centered at every point inside its disk of convergence. This fact is not obvious.

\begin{theorem}\label{TEO76} ({\sf Rearrangement}) Let $f(z)= \sum a_nz^n$ be convergent in $D(0;\rho)$, where $\rho>0$, and $z_0\in D(0;\rho)$. Then, there exists a complex sequence $(b_n)$ satisfying
\[f(z) = \sum b_n(z-z_0)^n \,,\ \  \ \textrm{for all} \ z \in D(z_0;\rho - |z_0|).\]
\end{theorem}
\begin{proof} Let us consider $z$ such that $|z_0| + |z - z_0|<\rho$. Since the given power series converges absolutely within $D(0;\rho)$, we conclude that the values of the two unordered sums
\[ \sum_n |a_n|\big(\,|z_0|+|z-z_0|\,\big)^n \ =\  \sum_n\sum_{0\leq p\leq n}|a_n|\binom{n}{p}|z_0|^{n-p}\,|z-z_0|^{p}  \]
are equal and finite. Thus, the family $\big(|a_n|\binom{n}{p}|z_0|^{n-p}\,|z-z_0|^{p}\big)$, where $n\in \mathbb N$ and $0\leq p\leq n$, is summable. Employing Proposition \ref{PRO61} and Lemma \ref{LE410} we deduce the identities
\begin{displaymath}
\sum\limits_n\sum\limits_{0\leq p\leq n}a_n\binom{n}{p}z_0^{n-p}\,(z-z_0)^{p}
= \left\{\begin{array}{ll}
\sum\limits_n a_n(z_0\,+\,z-z_0)^n = \sum\limits_{n}a_nz^n=f(z),\\
\\
\sum\limits_{p=0}^{+\infty}\left(\sum\limits_{n=p}^{+\infty}a_n\binom{n}{p}z_0^{n-p}\right)(z-z_0)^p.\\
\end{array} 
\right.
\end{displaymath}
Moreover, we find that  $b_p=\sum_{n=p}^{+\infty}a_n\binom{n}{p}z_0^{n-p}$, for all $p$ in $\mathbb N$.
\end{proof}

As is the case with polynomials, we can multiply power series. 
\begin{theorem}\label{TEO77} ({\sf Cauchy Product}) Let $\sum\limits a_nz^n$ and $\sum\limits b_nz^n$ be convergent within $D(0;r)$, where $r>0$. Then, we have 
\begin{displaymath}
\left\{\begin{array}{ll} 
(\,\sum\limits a_nz^n\,)(\,\sum\limits b_nz^n\,)= \sum c_n z^n,\ \textrm{for all}\ z \in D(0;r),\\
\\
\textrm{where}\ c_n = \sum\limits_{j+k =n}a_jb_k,\ \textrm{for all}\ n\in \mathbb N.
\end{array}
\right.
\end{displaymath}
\end{theorem}
\begin{proof} Let us fix $z \in D(0;r)$. From Theorem \ref{TEO72} follow that $\sum a_nz^n$ and $\sum b_nz^n$ are both absolutely convergent. Thus, by Corollary \ref{COR63} we conclude that 
\[(\sum\limits_na_nz^n)(\sum\limits_nb_nz^n)= \sum\limits_n\,(\sum\limits_{j +k =n}a_jz^jb_kz^k)= 
\sum\limits_n\,(\sum\limits_{j +k =n}a_jb_k)z^n .\]
\end{proof}

\begin{corollary}\label{COR78} ({\sf Pth Power)} Let $\sum a_nz^n$ be convergent in $D(0;r)$, with $r>0$, and $p \in\mathbb N$. Then, for every $z\in D(0;r)$ the following identity holds
\[\big( \sum a_nz^n\big)^p = \sum b_nz^n, \ \textrm{where}\ b_n=\sum_{n_1+ \cdots +n_p=n}a_{n_1} \ldots a_{n_p} . \]
\end{corollary}
\begin{proof} Let us fix $z \in D(0;r)$. It is clear that
\[\infty > \big(\,\sum|a_n||z|^n\,\big)^p = 
\sum_{n_1\in\, \mathbb N, \ldots ,n_p\in\,\mathbb N} |a_{n_1}||z|^{n_1} \ldots |a_{n_p}||z|^{n_p} .\]
Therefore, the family $(a_{n_1}z^{n_1} \ldots a_{n_p}z^{n_p})$, where $n_1, \ldots, n_p$ run over $\mathbb N$, is summable. Employing Corollary \ref{COR52} (associative law) we deduce the identities
\begin{displaymath}
\sum\limits_{n_1, \ldots, n_p}a_{n_1}z^{n_1} \ldots a_{n_p}z^{n_p} =
\left\{\begin{array}{ll}
\big(\sum\limits_{n_1}a_{n_1}z^{n_1}\big)\ldots \big(\sum\limits_{n_p}a_{n_p}z^{n_p}\big)\ =\ \big(\sum\limits_{n}a_nz^n\,\big)^p\\
\\
\sum\limits_n\Big(\sum\limits_{n_1 + \cdots + n_p=n}a_{n_1} \ldots a_{n_p}\Big)\,z^n.
\end{array}
\right.
\end{displaymath}
\end{proof}

Next, we show that analogously to polynomials we can compose power series.

\begin{theorem}\label{TEO79}{\sf (Composition)} Let $f(z) = \sum\limits a_nz^n$ and $g(z) =\sum\limits b_mz^m$ be two convergent power series within $D(0;R)$, with $R>0$. If $|g(0)|<R$, then there exists a complex sequence $(c_m)$ and $r >0$ such that
\[f(g(z)) = \sum c_mz^m,\ \textrm{for all}\ z \in D(0;r) . \]
\end{theorem}
\begin{proof} By Theorem \ref{TEO72}, the power series $\sum a_nz^n$ and $\sum b_mz^m$ are both absolutely convergent inside $D(0;R)$.
Let us choose $\rho$ such that $|g(0)|=|b_0|<\rho<R$. Then, fixing $z$ in the open disk $D(0;r)$, where $0<r< (\rho - |b_0|)/(1+\sum_{m\geq 1} |b_m|\rho^{m-1})$, we obtain the trivial inequalities $\sum_{m\geq 1}|b_m||z|^m\leq r\sum_{m\geq 1}|b_m|\rho^{m-1}< \rho -|b_0|$. Hence, we deduce that $\sum_{m\geq 0} |b_m||z|^m<\rho<R$ and
\[\infty>\sum_n|a_n|\Big(\sum_m |b_m|\,|\,z|^{\,m}\Big)^n  = \sum_{n\in \mathbb N}|a_n|\sum_{m_1\in \mathbb N, \ldots, m_n\in \mathbb N}|b_{m_1}||z|^{m_1} \ldots |b_{m_n}||z|^{m_n} .\]
Thus, the family $(a_n b_{m_1}z^{m_1} \ldots b_{m_n}z^{m_n})$, where $n, m_1,m_2, \ldots, m_n \in \mathbb N$, is summable.
Employing Corollary \ref{COR52} (associative law) we find the equalities
\begin{displaymath}
\ \sum_{n}a_n\sum_{m_1, \ldots ,m_n}b_{m_1}z^{m_1} \ldots b_{m_n}z^{m_n}
=  \left\{\begin{array}{ll}
\sum\limits_n a_n\big(\sum\limits_m b_mz^m\big)^n  =  f\big(g(z)\big)\\
\\
\sum\limits_m\Big(\sum\limits_n\,a_n\sum\limits_{m_1+ \cdots +m_n=m}b_{m_1}\ldots b_{m_n}\Big)z^m.
\end{array}
\right.
\end{displaymath}
\end{proof}

The following proof is standard.
 
\begin{proposition}\label{PRO710}({\sf Reciprocal)} Let $f(z)= \sum a_nz^n$ be convergent in $ D(0;r)$, where $r>0$ and $a_0\neq 0$. Then, there exists $\delta >0$ and a complex sequence $(b_n)$ such that 
\[\frac{1}{f(z)} = \sum b_nz^n,  \ \textrm{for all}\  z \in D(0;\delta) .\] 
\end{proposition}
\begin{proof}
Without loss of generality we can assume that $a_0=1$. Hence, the functions $g(z)= 1 -f(z)= -\sum_{n\geq 1}a_n z^n$  and $h(z)= (1-z)^{-1}= \sum_{n\geq 0}z^n$ are both defined in an open disk containing the origin and $g(0)=0$. Thus, by Theorem ~\ref{TEO79} there exists $\delta>0$ such that the composition function $h(g(z))=\{1 -[1 -f(z)]\}^{-1}=1/f(z)$ can be developed as a convergent power series inside $D(0;\delta)$.
\end{proof}

\section{Power Series - Analytic Properties}

\begin{definition}\label{DEF81} Let $f: \Omega \to \mathbb C$ be defined in an open set $\Omega \subset \mathbb C$ and $z_0 \in \Omega$. We say that $f$ is \textit{complex-differentiable} at a point $z_0\in \Omega$ if the limit  
\[ f'(z_0) = \lim\limits_{z\to z_0}\frac{f(z) - f(z_0)}{z- z_0}\]
exists. The number $f'(z_0)$ is the derivative of $f$ at $z_0$.
\end{definition}

\begin{remark} It is easy to prove that if $f'(z_0)$ exists, then $f$ is continuous at $z_0$. 
\end{remark}
For practical purposes, let us write the power series $\sum_{n\geq 1} na_nz^{n-1}$ as $\sum na_nz^{n-1}$. The next result shows that, as is the case with polynomials, power series can be differentiated term by term.

\begin{theorem} \label{TEO83}{\sf (Differentiation)} The power series $f(z) = \sum a_nz^n$ and $g(z) = \sum na_nz^{n-1}$ have same disk of convergence $D(0;\rho)$. If $\rho>0$, then we have 
\[f'(z) = g(z), \ \textrm{for all}\ z \in D(0;\rho).\]
\end{theorem}

\begin{proof} Let us split the proof into two parts that complement each other.
\begin{itemize}
\item[(1)] From the inequality $\sum_{n\geq 1} |a_nz^n|\leq |z|\sum_{n\geq 1}|na_nz^{n-1}|$ it follows that the disk of convergence of $g$ is contained in the disk of convergence of $f$. Thus, if the disk of convergence of $f$ degenerates then so does the disk of convergence of $g$.

\item[(2)] Now, let us suppose that $f$ is convergent in $D(0;\tau)$, where $\tau>0$. Let us fix $R>0$ and $z\in \mathbb C$ such that $|z|<R<\tau$. Let us also consider an arbitrary $h\in \mathbb C$ satisfying $0<|h|<r=R-|z|$. Hence, given $n\in\{2,3,4,...\}$ we have 
\[\frac{(z+h)^n -z^n}{h} = nz^{n-1} \, +\, h\sum\limits_{p=2}^n\binom{n}{p}z^{n-p}h^{p-2}\]
and 
\[\left|\frac{(z+h)^n -z^n}{h}\ -\ nz^{n-1}\right|  \, \leq \, \frac{|h|}{r^2}\ \sum_{p=2}^n\,\binom{n}{p}\,|z|^{n-p}\,r^p\  \leq \frac{|h|\,}{r^2}R^n .\]
From the above inequality and the choices of $R$, $z$, and $h$ it follows that $\sum na_nz^{n-1}$ converges absolutely. Thus, $g$ converges in $D(0;\tau)$. Moreover,
\[\left|\,\sum a_n\frac{(z+h)^n - z^n}{h}\,-\, \sum na_nz^{n-1}\,\right| \, \leq \, \frac{|h|}{r^2}\sum|a_n|R^n\ \xrightarrow{h\to 0}\ 0.\] 
\end{itemize}
\end{proof}

\begin{corollary}\label{COR84} Let $f(z) = \sum a_nz^n$ be convergent in $D(0;\rho)$, with $\rho >0$. Then, $f$ is infinitely  differentiable and the following assertions are true.
\begin{itemize}
\item[(a)] $f^{(k)}(z) = \sum\limits_{n\geq k}n(n-1) \ldots (n-k+1)a_nz^{n-k} =\sum\limits_{n\geq k}\frac{n!}{(n-k)!}a_nz^{n-k}$.
\item[(b)] $f$ is given by its {\sf Taylor series} centered at the origin,
\[f(z)= \sum_{n\geq 0}\frac{f^{(n)}(0)}{n!}z^n.\]
\item[(c)] ({\sf Uniqueness Theorem for the Coefficients}) If $f$ is identically zero, then all the coefficients $a_n$, where $n\in \mathbb N$, are zero.
\end{itemize}
\end{corollary}
\begin{proof} From Theorem \ref{TEO83} follow that $f$ is infinitely differentiable and item (a).
\begin{itemize}
\item[(b)] Substituting $z=0$ into (a) we obtain $f^{(k)}(0)=k!a_k$. Thus, $a_k=f^{(k)}(0)/k!$.
\item[(c)] It follows from (b).
\end{itemize}
\end{proof}

\begin{remark} For a proof of the uniqueness theorem for the coefficients of a complex power series [Corollary ~\ref{COR84} (c)] that employs neither differentiability nor function continuity, we refer the reader to de Oliveira ~\cite[Remark 3.5]{OO3}.
\end{remark}

The next result (Lemma ~\ref{LEM86}) will be employed in the proof of Corollary ~\ref{COR87}.

\newpage

\begin{lemma} \label{LEM86} Let $f(z)=\sum a_nz^n$ be convergent in $D(0;\rho)$, where $\rho>0$. Let us suppose that $f(z)=0$ for all $z$ in a disk $D(z_0;\epsilon)$, where $\epsilon >0$, contained in $D(0;\rho)$. Then, f vanishes everywhere.
\end{lemma}
\begin{proof} By Corollary \ref{COR84}, it is enough to show that $f$ is identically zero over an open disk containing the origin. If $0\in D(z_0;\epsilon)$, then the proof is complete. Hence, let us suppose that $0\notin D(z_0;\epsilon)$ and fix the segment $\overline{0z_0}$, joining $0$ and $z_0$. 

The point $z_1= z_0 -\epsilon \frac{z_0}{|z_0|}$ belongs to the segment $\overline{0z_0}$ and satisfies $d(z_1;0)=|z_0|-\epsilon$. From Theorem \ref{TEO76} (the Rearrangement Theorem) we conclude that there exists a power series centered at $z_1$ and punctually convergent to $f$, at every point inside $D(z_1;\epsilon)$. Since $f$ and its derivatives are continuous and vanish at every point inside $D(z_0;\epsilon)$, we deduce that $f^{(n)}(z_1)=0$ for all $n\in \mathbb N$. Thus, by Corollary ~\ref{COR84}, $f$ is identically zero over $D(z_1;\epsilon)$. If $0\in D(z_1;\epsilon)$, then the proof is complete. 

If $0\notin D(z_1;\epsilon)$, by proceeding as in the argument above a finite number of times we find a finite sequence of points $z_j= z_0 - j\epsilon\frac{z_0}{|z_0|}$, where $j\in \{0,1,\ldots, N\}$, within the segment $\overline{0z_0}$ and satisfying the following four conditions: $f$ is identically zero over $D(z_j;\epsilon)$ for all $j\in \{0,1,\ldots, N\}$, $d(z_j;0)= |z_0|- j\epsilon \geq \epsilon$ for each $j\in \{0,1,\ldots, N-1\}$, $d(z_N;0)<\epsilon$, and  $f^{(n)}(z_j)=0$ for all $n\in \mathbb N$ and all $j\in\{0,1,\ldots,N\}$.  Since $0\in D(z_N;\epsilon)$, the proof is complete. 
\end{proof}

\vspace{0,2 cm}

The following corollary shows that, analogously to polynomials, nonzero power series have isolated zeros.

\begin{corollary}\label{COR87} ({\sf Principle of Isolated Zeros)} Let
$f(z)=\sum a_nz^n$ be a power series convergent within $D(0;\rho)$, where $\rho>0$, such that $f(z_0)=0$ for some $z_0\in D(0;\rho)$ but $f$ is not the zero function. Then,  there exists a smallest $k \geq 1$ satisfying $f^{(k)}(z_0)\neq 0$ and also a convergent power series $g(z)=\sum b_n(z-z_0)^n$ within $D(z_0;\delta)$, for some $\delta >0$, such that we have the factorization
\[ f(z) =(z-z_0)^kg(z), \ \textrm{for all}\ z \in
D(z_0;\delta), \ \textrm{with}\  g\ \textrm{nowhere vanishing}.\]  
\end{corollary}
\begin{proof} Employing Theorem \ref{TEO76} (the Rearrangement Theorem) and Corollary \ref{COR84} we can write the development $f(z) = \sum_{n=0}^{+\infty} f^{(n)}(z_0)(z-z_0)^n/n!$, where $z \in D(z_0;\rho')$, for some $\rho'>0$. Since $f$ is not identically zero over $D(0;\rho)$, from Lemma ~\ref{LEM86}  we infer that $f$ is not identically zero over $D(z_0;\rho')$. Therefore, since $f^{(0)}(z_0)=f(z_0)=0$, there exists the smallest $k\geq 1$ such that $f^{(k)}(z_0)\neq 0$. Putting $b_n= f^{(n)}(z_0)/n!$, where $n\in\mathbb N$, we arrive at
\[f(z)=\sum\limits_{n= 0}^{+\infty}b_n(z-z_0)^n= (z-z_0)^k\big[b_k + b_{k+1}(z-z_0) + b_{k+2}(z-z_0)^2 + \cdots \big], \ \textrm{if}\ z \in D(z_0;\rho').\]
The function $g(z) = \sum_{j=0}^{+\infty}b_{k+j}(z-z_0)^j$, where $z \in D(z_0;\rho')$, satisfies 
 \[f(z)=(z-z_0)^k g(z), \ \textrm{for all}\ z \in D(z_0;\rho'), \  \textrm{and} \  g(z_0)=b_k \neq 0.\]
Finally, we choose $\delta$, with $0<\delta <\rho'$, such that we have $g(z)\neq 0$ if $z \in D(z_0;\delta)$.
\end{proof}

\newpage

It is well-known that a polynomial of order $n$ is determined by its values at $n+1$ distinct points. The following result is a similar one, for power series.

\begin{corollary} \label{COR88}({\sf Identity Principle)} Let $\sum a_nz^n$ and $\sum b_nz^n$ be convergent in $D(0;r)$. Let $X$ be a subset of $D(0;r)$ such that $X$ has an accumulation point in $D(0;r)$. If the identity $\sum a_nz^n=\sum b_nz^n$ holds for all $z \in X$, then we have $a_n=b_n$ for all $n\in \mathbb N$. 
\end{corollary}
\begin{proof} The subtraction $\sum a_nz^n - \sum b_nz^n = \sum (a_n-b_n)z^n$ shows that without loss of generality we can assume $b_n=0$, for all $n \in \mathbb N$. Hence, let us suppose that $z_0$ is an accumulation point of $X=\{z: f(z)=\sum a_nz^n=0\}$. By the continuity of $f$ we obtain $f(z_0)=0$. If $f$ is not the zero function, then through the principle of isolated zeros (Corollary \ref{COR87}) we deduce that $z_0$ is the only zero of $f$ in some $D(z_0;r')$, where $r'>0$, contradicting our assumption about $z_0$. Therefore, we have $f(z) = \sum a_nz^n = 0$, for all $z \in D(0;r)$, and $a_n=f^{(n)}(0)/n!=0$, for all $ n \in \mathbb N$. 
\end{proof}

Next, we prove a particular result about the  complex binomial series.

\begin{proposition}\label{PROP89} ({\sf Binomial Series}) Let $p\in \mathbb N\setminus\{0\}$. Then $B(z)= \sum_{n=0}^{+\infty}\binom{1/p}{n}z^n$ converges in the open disk $D(0;1)$ and $B(z)$ is a pth root of $1+z$, with $z \in D(0;1)$. That is, we have $B(z)^p=1+z$ for every $z\in D(0;1)$. 
\end{proposition}
\begin{proof} If $x$ is a real number, with $-1< x<1$, then the validity of the formula $b(x)=(1+x)^\frac{1}{p}=\sum_{n=0}^{+\infty}\binom{1/p}{n}x^n$ is 
well-known. It is also known that the real series $\sum_{n=0}^{+\infty}\binom{1/p}{n}x^n$ diverges if $|x|>1$. Hence, by Theorem ~\ref{TEO72} (a) the complex power series
$B(z)=\sum_{n=0}^{+\infty}\binom{1/p}{n}z^n$, where $z \in \mathbb C$, has radius of convergence $\rho=1$. By Corollary ~\ref{COR78}, 
the function $B(z)^p$ is expressible as a convergent power series inside $D(0;1)$. Clearly, we have $B(x)^p=b(x)^p=1+x$, if $-1<x<1$. The claim therefore follows from the identity principle (Corollary ~\ref{COR88}).
\end{proof}

As a final result, we enunciate the {\sf Inverse Function Theorem for Power Series}. {\sl Let us suppose that $f(z)=\sum_{n=1}^{+\infty}a_nz^n$ converges in $D(0;\rho)$, where $\rho >0$ and $f'(0)=a_1\neq 0$. Then, there exists a unique power series $g(z)=\sum_{m=1}^{+\infty}b_mz^m$ that converges in some disk $D(0;\delta)$, with $\delta >0$, and satisfies $f(g(z))=z$, for all $z \in D(0;\delta)$.} For a power series proof of this theorem we refer the reader to Knopp ~\cite[pp.~184--188]{KK} (see also Cartan ~\cite[pp.~26--27]{CA} and Lang ~\cite[pp.~76--79]{LA}). Keeping all the hypothesis on the power series $f(z)$, for a rather short and easy power series proof that {\sl $f$ is inversible in a small disk $D(0;\rho')$, for some $\rho' >0$, and its correspondent local inverse is a complex-differentiable function at every point in its domain}, we refer the reader to de Oliveira ~\cite[Theorem 8.1]{OO3}.

\

\paragraph{\bf Acknowledgments.} I would like to sincerely thank Professors J. Aragona, Paulo A. Martin, and R. B. Burckel for their comments and suggestions. Any possible slips and mistakes are my responsibility.

\bibliographystyle{amsplain}

\end{document}